\documentclass[]{article}

\usepackage[utf8]{inputenc}
\usepackage[english]{babel}
\usepackage{pdfpages}
\usepackage[a4paper]{geometry}
\usepackage{graphicx}
\usepackage{amsmath,amsthm,amssymb}
\usepackage{wasysym}
\usepackage{microtype}
\usepackage[colorlinks=false, pdfborder={0 0 0}]{hyperref}
\usepackage[capitalise]{cleveref}

\title{Coarse fundamental groups and box spaces}
\author{Thiebout Delabie\thanks{Universit\'{e} de Neuch\^{a}tel, \url{thiebout.delabie@unine.ch}, supported by Swiss NSF grant no. 200021 163417} 
\ and Ana Khukhro\thanks{Universit\'{e} de Neuch\^{a}tel, \url{anastasia.khukhro@unine.ch}}}

\date{\today}

\newtheorem{theorem}{Theorem}[section]
\newtheorem{lemma}[theorem]{Lemma}
\newtheorem{proposition}[theorem]{Proposition}

\newtheorem{definition}[theorem]{Definition}
\newtheorem*{acknow}{Acknowledgements}

\newcommand{\Cay}{\operatorname{Cay}}
\newcommand{\N}{\mathbb{N}}
\newcommand{\Z}{\mathbb{Z}}

\newcommand{\close}[1]{$#1$-close}
\newcommand{\homotopic}[1]{$#1$-homotopic}
\newcommand{\SL}[2]{\operatorname{SL}(#1,#2)}
\newcommand{\rk}[1]{\operatorname{rk}(#1)}
\newcommand{\RG}[2]{\operatorname{RG}(#1,#2)}

\parskip\bigskipamount

\begin{document}

\maketitle

\begin{abstract}
We use a coarse version of the fundamental group first introduced by Barcelo, Kramer, Laubenbacher and Weaver to show that box spaces of finitely presented groups detect the normal subgroups used to construct the box space, up to isomorphism. As a consequence we have that two finitely presented groups admit coarsely equivalent box spaces if and only if they are commensurable via normal subgroups. We also provide an example of two filtrations $(N_i)$ and $(M_i)$ of a free group $F$ such that $M_i>N_i$ for all $i$ with $[M_i:N_i]$ uniformly bounded, but with $\Box_{(N_i)}F$ not coarsely equivalent to $\Box_{(M_i)}F$. Finally, we give some applications of the main theorem for rank gradient and the first $\ell^2$ Betti number, and show that the main theorem can be used to construct infinitely many coarse equivalence classes of box spaces with various properties. 
\end{abstract}

\section{Introduction}
Given a free group $G=F_S$ on a set $S$, and the Cayley graph of a quotient $G/N$ of this free group with respect to the image of $S$, one can recover the subgroup $N$ by computing the fundamental group of this Cayley graph, which will be isomorphic to the free group $N$. If the group $G$ is not free, one cannot use the fundamental group in this way.

In this paper, we use a coarse version of the fundamental group, first defined for simplicial complexes in \cite{BKLW} (and then in full generality in \cite{BCW}, see also \cite{BBLL}, \cite{BL}), to recover normal subgroups used to construct quotients in the context of box spaces of finitely presented groups.

\begin{definition}
Given a residually finite, finitely generated group $G$, a filtration of $G$ is a sequence of nested, normal, finite index subgroups $(N_i)$ of $G$ such that the intersection $\cap_i N_i$ is trivial. Let $S$ be a fixed generating set of $G$. The box space $\Box_{(N_i)}G$ of $G$ with respect to a filtration $(N_i)$ is the disjoint union $\sqcup_i G/N_i$ metrized using the Cayley graph metric induced from $S$ on each component, such that the distance between any two distinct components is equal to the sum of their diameters. 
\end{definition}

In this way, box spaces are coarse geometric objects which encode information about a residually finite group. There are many fascinating connections between the geometric properties of a box space and the group used to construct it (see for example \cite{CWW}, \cite{Mar}, \cite{NY}, \cite{Roe}, \cite{WY}). For this reason, box spaces are often the first examples or counterexamples of spaces with interesting properties in coarse geometry.

The natural notion of equivalence between metric spaces in the world of coarse geometry is defined as follows.

\begin{definition}
Given two metric spaces $(X,d_X)$ and $(Y,d_Y)$, a map $\varphi:X \rightarrow Y$ is a coarse embedding of $X$ into $Y$ if there exist proper functions $\rho_{\pm}: \mathbb{R}_+ \rightarrow \mathbb{R}_+$ such that for all $a,b \in X$,
$$\rho_-(d_X(a,b))\leq d_Y(\varphi(a),\varphi(b)) \leq \rho_+(d_X(a,b)).$$
The map $\varphi$ is a coarse equivalence if, in addition, there exists $C>0$ such that for all $y\in Y$, the image of $X$ is at most distance $C$ from $y$. If there exists such a $\varphi$, we say that $X$ and $Y$ are coarsely equivalent, and write $X \simeq_{CE} Y$. 
\end{definition}

For example, choosing different generating sets in the construction of box spaces gives rise to coarsely equivalent spaces. It is interesting to investigate the rigidity of box spaces, i.e. to which extent the coarse equivalence class of a box space determines the group it was constructed from. Two results in this direction are as follows. We write $X\simeq_{\text{QI}} Y$ to mean that two spaces $X$ and $Y$ are quasi-isometric (that is, they are coarsely equivalent, and there exists a constant $C>0$ such that $\rho_-(t)= t/C -C$ and $\rho_+(t)= Ct +C$).

\begin{theorem}[\cite{AA}]
Let $G$ and $H$ be finitely generated groups with respective filtrations $N_i$ and $M_i$. If $\Box_{(N_i)}G\simeq_{\text{CE}} \Box_{(M_i)}H$, then $G\simeq_{\text{QI}} H$. 
\end{theorem}

\begin{theorem}[\cite{Das}]
Given two finitely generated groups $G$ and $H$ with respective filtrations $N_i$ and $M_i$, $\Box_{(N_i)}G\simeq_{\text{CE}} \Box_{(M_i)}H$ implies that $G$ and $H$ are uniformly measure equivalent.
\end{theorem}

An application of these results is being able distinguish box spaces up to coarse equivalence, which one can for example use as in \cite{AA} to show that there exist uncountably many expanders with geometric property (T) of Willett and Yu (\cite{WY}). We note that these results were recently generalised to sofic approximations of groups (\cite{AFS}). 

In this paper we will prove the following strong rigidity theorem for box spaces of finitely presented groups.
\begin{theorem}\label{main}
Let $G$ be a finitely presented group and $H$ a finitely generated group, with respective filtrations $N_i$ and $M_i$ such that $\Box_{(N_i)}G\simeq_{\text{CE}} \Box_{(M_i)}H$. Then there exists an almost permutation with bounded displacement $f$ of $\N$ such that $N_i$ is isomorphic to $M_{f(i)}$ for every $i$ in the domain of $f$.
\end{theorem}

Here, as in \cite{AA}, an almost permutation $f$ of $\N$ is a bijection between two cofinite subsets of $\N$. We say that $f$ has bounded displacement if there exist $A$ and $N$ such that $|f(n)-n|\leq A$ for all $n\geq N$.

We achieve this rigidity result by using the coarse fundamental group of \cite{BKLW}.
The idea of using the coarse fundamental group in this context comes from the following intuitive idea: if $G= \left\langle S|R \right \rangle$ is a finitely presented group, and $N$ is a normal subgroup of $G$ such that non-trivial elements in $N$ are ``sufficiently long'' with respect to the generating set of $G$, then one can use the coarse fundamental group at a scale which lies between the length of the longest element in $R$ and the shortest element in $N$, so that only loops coming from $N$ are detected in $\Cay(G/N)$.

This idea works well for eventually detecting the normal subgroups used to construct box spaces, since the condition that a filtration $(N_i)$ must be nested and have trivial intersection implies that the length of non-trivial elements in the $N_i$ tends to infinity as $i$ tends to infinity.

The main theorem can be applied in various contexts. In particular, we deduce the following rigidity result, and give a surprising example of box spaces which are algebraically ``close'', but not coarsely equivalent.

\begin{theorem}\label{ComInt}
Let $G$ be finitely presented group and $H$ a finitely generated group. Then there exist two filtrations $N_i$ and $M_i$ of $G$ and $H$ respectively such that $\Box_{(N_i)} G$ and $\Box_{(M_i)} H$ are coarsely equivalent if and only if $G$ is commensurable to $H$ via a normal subgroup.
\end{theorem}

\begin{theorem}\label{bounded}
There exist two box spaces $\Box_{(N_i)}G\not\simeq_{\text{CE}}\Box_{(M_i)}G$ of the same group $G$ such that $G/N_i\twoheadrightarrow G/M_i$ with $[M_i:N_i]$ bounded.
\end{theorem}

We deduce the following for rank gradient and the first $\ell^2$ Betti number.

\begin{theorem}\label{RGInt}
Given two finitely presented, residually finite groups $G$ and $H$ and respective filtrations $(N_i)$ and $(M_i)$, if $\Box_{(N_i)}G$ is coarsely equivalent to $\Box_{(M_i)}H$ then $\RG{G}{(N_i)}>0$ if and only if $\RG{H}{(M_i)}>0$. 
\end{theorem}

\begin{theorem}\label{l2Int}
Given two finitely presented, residually finite groups $G$ and $H$ and respective filtrations $(N_i)$ and $(M_i)$, if $\Box_{(N_i)}G$ is coarsely equivalent to $\Box_{(M_i)}H$ then $\beta_1^{(2)}(G)>0$ if and only if $\beta_1^{(2)}(H)>0$.
\end{theorem}

We also use our result to coarsely distinguish various box spaces of interest, as follows.

\begin{theorem}\label{HilbInt}
For each $n\geq 2$, there exist infinitely many coarse equivalence classes of box spaces of the free group $F_n$ that coarsely embed into a Hilbert space.
\end{theorem}

\begin{theorem}\label{InfRam}
There exist infinitely many coarse equivalence classes of box spaces of the free group $F_3$ that contain Ramanujan expanders.
\end{theorem}

\begin{theorem}\label{coprimeInt}
Given $n \geq 3$, there exists a box space of the free group $F_n$ such that no box space of $F_m$ with $m-1$ coprime to $n-1$ is coarsely equivalent to it.
\end{theorem}

We also prove the following, which corrects Proposition 4 of \cite{Kh12}.

\begin{theorem}\label{CorInt}
There exists a finitely generated, residually finite group $G$ with filtration $(M_i)$ and a normal finite index subgroup $H\lhd G$ such that $\Box_{(M_i)}G$ and $\Box_{(M_i\cap H)}H$ are not coarsely equivalent.
\end{theorem}

\begin{acknow}
We thank Alain Valette for his valuable comments, and for providing a nice proof of the existence of infinitely many primes with the required properties in \cref{RamThm}. The first author was supported by Swiss NSF grant no. 200021 163417.
\end{acknow}

\section{Preliminaries}

Throughout this paper, we will identify paths in a Cayley graph $\Cay(G,S)$ with words in the elements of $S$. In this section, we explain in detail how this works. We also recall basic properties of classical fundamental groups, and give the definition of the coarse fundamental group from \cite{BCW}. 

Given a connected graph $X$, the natural metric to consider on the vertices of $X$ is the shortest path metric, that is, the distance between two vertices $x$ and $y$ is the number of edges in a shortest path of edges (i.e. a sequence of adjacent edges) linking $x$ to $y$. Such a path that realises the distance between $x$ and $y$ is called a geodesic between $x$ and $y$.

In this paper, we consider Cayley graphs, which are graphs arising from groups. 
Given a finitely generated group $G= \left\langle S \right\rangle$, $\Cay(G,S)$ denotes the Cayley graph of $G$ with respect to $S$, which has $G$ as its vertex set, and $\{(g,gs): g\in G, s\in S\}$ as its edge set. An edge $(g, gs)$ is said to be labelled by the generator $s\in S$. Note that given a path from the trivial element $e\in G$ to $g\in G$, the sequence of labels of edges in this path gives a word in the generating set $S$ which is equal to the element $g$ in $G$. In this way, words read along loops in the Cayley graph are relators in $G$ (i.e. they are equal to the trivial element $e$ in $G$). A word is said to be reduced if it does not contain occurrences of $ss^{-1}$ or $s^{-1}s$ for $s\in S$, or $e$ if $e\in S$. 
We will omit the generating set in the notation $\Cay(G)$ when there is no ambiguity.

Given a free group $F_S$ on a finite set $S$, consider the Cayley graph $(\Cay(F_S/N), \bar{S})$ of a quotient of $F_S$ by a normal subgroup $N$, with respect to the image of the generating set $S$. One can pick a basepoint (say, the identity element) in this graph, and consider homotopy classes of loops based at this basepoint. These homotopy classes form a group (with the composition of elements being concatenation of loops) called the fundamental group, denoted $\pi_1(\Cay(F_S/N))$. 

In this situation, this group is isomorphic to the subgroup $N$ of $F_S$. Note that it is a free group (the fundamental group of a graph is always free), generated by the labels of the loops representing the homotopy classes of loops that generate $\pi_1(\Cay(F_S/N))$. In fact, since we can view the free group $F_S$ as the fundamental group of a wedge $W_S$ of $|S|$ circles labelled by elements of $S$, we can recover the graph $\Cay(F_S/N)$ by taking the covering space of this wedge of circles corresponding to the quotient $\pi_1(W_S)\cong F_S \rightarrow F_S/N$. The fundamental group of this covering space will be isomorphic to $N$ (see Section 1.A of \cite{Hat} for more on this Galois-type correspondence between covering spaces of graphs and subgroups of their fundamental groups).

If one now considers $\Cay(G/N)$ for a general finitely generated group $G$, it is no longer possible to detect the normal subgroup $N$ using $\pi_1(\Cay(G/N))$. This is because the loops in a Cayley graph represent relations (i.e. words in the generating set which represent the trivial element $e\in G$) and so relators of $G$ will also contribute to $\pi_1(\Cay(G/N))$. What one obtains as $\pi_1(\Cay(G/N))$ in this case is the kernel of the surjective homomorphism from the free group on the generating set of $G$ to the group $G/N$ (so it is again a free group). 

In order to try and detect the normal subgroup $N$ from the graph $\Cay(G/N)$, we will use a coarse version of the fundamental group first defined in \cite{BKLW} (see also \cite{BCW}). Informally, this is a fundamental group at a scale which will ignore small loops but detect big ones.

We will employ this notion in the context of finitely presented groups, i.e. groups $G$ which can be written as a quotient of a finitely generated free group $F_S$ by a subgroup which is the normal closure in $F_S$ of a finite set of elements, $R$. We write $G= \left\langle S|R \right \rangle$, and refer to $\left\langle S|R \right \rangle$ as the (finite) presentation of $G$.

We will now define a coarse version of the fundamental group for simple connected graphs. However, we will immediately restrict ourselves to Cayley graphs, as there we can use the neutral element as a base point for our fundamental group.

We first need a coarse notion of what it means for two paths to be homotopic.
Classically, two paths are homotopic if we can deform one path into the other in a continuous way. In the coarse setting, we take quasi-paths and we will make these deformations discrete.
This will depend on a constant $r>0$. An $r$-path $p$ in a metric space $X$ is an $r$-Lipschitz map $p\colon \{0,\ldots,\ell(p)\}\to X$. 
Note that a path in a graph is a $1$-path.

The deformation of the paths will also depend on the constant $r$, and if such a deformation exists, we will call the two paths \close{r}.
We say that two paths $p$ and $q$ in a graph $\mathcal{G}$ are \close{r} if one of the two following cases is satisfied:
\begin{enumerate}
\item[(a)]
For every $i\le \min(\ell(p),\ell(q))$ we have that $p(i)=q(i)$ and for bigger $i$ we either have $p(i)=p(\ell(q))$ or $q(i)=q(\ell(p))$, depending on which path is defined at $i$.
\item[(b)]
We have that $\ell(p)=\ell(q)$ and for every $0\le i\le \ell(p)$ we have $d(p(i),q(i))\le r$.
\end{enumerate}
Now we define a coarse version of homotopy by combining these deformations.
\begin{definition}
Let $\mathcal{G}$ be a graph, let $r>0$ be a constant and let $p$ and $q$ be two $r$-paths in $\mathcal{G}$. We say that $p$ and $q$ are \homotopic{r} if there exists a sequence $p_0=p$, $p_1$,\ldots, $p_n=q$ such that $p_i$ is \close{r} to $p_{i-1}$ for every $i\in\{1,2,\ldots, n\}$.
\end{definition}

We can now define the fundamental group up to $r$-homotopy as the group of $r$-homotopy equivalence classes of $r$-loops rooted in a basepoint, with the group operation corresponding to concatenation of loops.

\begin{definition}
The fundamental group up to $r$-homotopy $A_{1,r} (\mathcal{G},x)$ is defined to be the group of equivalence classes of $r$-loops rooted in a basepoint $x$ with the operation $\ast\colon A_{1,r} (\mathcal{G},x)^2\to A_{1,r} (\mathcal{G},x)\colon ([p],[q]) \mapsto [p\ast q]$ with $$p\ast q\colon \{0,1,\ldots, \ell(p)+\ell(q)\}\to \mathcal{G}\colon \left\{ \begin{array}{cccl}i&\mapsto& p(i)& \text{if }0\le i\le \ell(p), \\ i&\mapsto& q(i-\ell(p))& \text{if }\ell(p)+1\le i\le \ell(p)+\ell(q). \end{array}\right.$$
\end{definition}

As we will focus on the case where $\mathcal{G}$ is a Cayley graph $\Cay(G,S)$, we will write $A_{1,r}(\mathcal{G})=A_{1,r}(\mathcal{G},e)$, since the basepoint will always be taken to be the identity element.

\section{Coarse homotopy}

In this section we give some properties of paths with respect to $r$-homotopy, and then use these to prove our main results.

\subsection{$r$-homotopic paths}

We will now prove some elementary propositions about $1$-paths, and then show that for a given $r$-path, there is a $1$-path which is $r$-homotopic to it. We do this because it will be useful to consider $1$-paths in the proof of the main theorem. We remark that a $1$-path is also an $r$-path for $r\geq 1$. 

The first proposition shows that we can remove any backtracking (including staying at the same vertex) of a $1$-path. Note that given a 1-path in the Cayley graph, there is a corresponding word in $S\cup\{e\}$, which is the word read along the labelled edges of the path (with an occurrence of $e$ denoting staying at a vertex).
\begin{proposition}\label{reduce}
Let $G=\langle S\rangle$ be a group and let $r\ge1$. Let $p$ be a $1$-path in $\Cay(G)$, let $w$ be the word in elements of the set $S\cup\{e\}$ corresponding to $p$ and let $q$ be the path corresponding to the reduced version of $w$. Then $q$ is \homotopic{r} to $p$.
\end{proposition}
\begin{proof}
Suppose that $w$ is a word that is not in its reduced form, then it can be written as $w=u_1ss^{-1}u_2$ or $w=u_1eu_2$, where $u_1$ and  $u_2$ are words in the elements of $S\cup\{e\}$ and $s\in S$.
By induction it suffices to show that the $1$-path corresponding to $u_1u_2$ is \homotopic{r} to $p$, because we can reduce $w$ by removing these $ss^{-1}$ and $e$ until $w$ is in its reduced form.

In the case that $w=u_1eu_2$ we see that $p$ is \homotopic{r} to the $1$-path corresponding to the word $u_1u_2e$, call this path $\tilde{q}$. Then removing $\tilde{q}(\ell(\tilde{q}))$, we recover the $1$-path corresponding to $u_1u_2$, and thus we see that the 1-path corresponding to $w$ is $r$-homotopic to the 1-path corresponding to $u_1u_2$.

In the case that $w=u_1ss^{-1}u_2$, we see that $p$ is \homotopic{r} to the $1$-path corresponding to the word $u_1e^{2}u_2$, call this path $q'$
(in fact $p$ and $q'$ are \close{r}, because $\ell(p) = \ell(u_1)+2+\ell(u_2) = \ell(q')$ and $d(p(i),q'(i))\le r$). Then, by inductively using the other case, we show that $p$ is \homotopic{r} to the $1$-path corresponding to $u_1u_2$.
\end{proof}

The second proposition shows that $r$-homotopies can pass holes of ``size'' $2r$.
\begin{proposition}\label{jump}
Let $G$ be a group and let $r\ge 1$. Let $u$, $v$ and $w$ be three words in the elements $S$ such that $uw$ and $v$ correspond to $e$ in $G$ and $\ell(v)\le 2r$. Then the $1$-paths corresponding to $uw$ and $uvw$ are \homotopic{r}.
\end{proposition}
In fact it is possible to show that the proposition still holds for holes of size $4r$ (i.e. $\ell(v)\le 4r$), but we restrict ourselves to $\ell(v)\le 2r$ to avoid unnecessarily complicating the proof.
\begin{proof}
As $\ell(v)\le 2r$ the $1$-path corresponding to $uvw$ is \homotopic{r} to the one corresponding to $ue^{\ell(v)}w$ and due to \cref{reduce} they are \homotopic{r} to the $1$-path corresponding to $uw$.
\end{proof}

Finally, we prove a proposition which will allow us to work with $1$-paths in the next section.
\begin{proposition}\label{path}
Let $\mathcal{G}$ be a graph, let $r\ge 1$ and $p$ be an $r$-path in $\mathcal{G}$, then there exists a $1$-path $q$ in $\mathcal{G}$ which is \homotopic{r} to $p$.
\end{proposition}
\begin{proof}
If $\ell(p)=0$, then $p$ is already a $1$-path.

If $\ell(p)\ge 1$, then by induction we can assume that $d(p(i),p(i+1))\le 1$ for $i\ge 1$. As $\mathcal{G}$ is a graph we can take a geodesic $1$-path $g$ between $p(0)$ and $p(1)$. Then take $$q\colon \{0,\ldots,\ell(p)+\ell(g)\}\to \mathcal{G}\colon \left\{ \begin{array}{cccl}i&\mapsto& g(i)& \text{if }0\le i\le \ell(g) \\ i&\mapsto& p(i-\ell(g))& \text{if }\ell(g)+1\le i\le \ell(p)+\ell(g). \end{array}\right.$$
Now $q$ is a $1$-path and since $\ell(g)\le r$, it is \close{r} to
$$\tilde{q}\colon \{0,\ldots,\ell(p)+\ell(g)\}\to \mathcal{G}\colon \left\{ \begin{array}{cccl}i&\mapsto& p(i)& \text{if }0\le i\le \ell(p) \\ i&\mapsto& p(\ell(p))& \text{if }\ell(p)+1\le i\le \ell(p)+\ell(g), \end{array}\right.$$
which is \close{r} to $p$. Thus $q$ is \homotopic{r} to $p$.
\end{proof}

\subsection{Box spaces of finitely presented groups}
In this section we will prove the main result, that box spaces of finitely presented groups eventually detect the normal subgroups used to construct them, even when we look up to coarse equivalence, i.e.\ given the coarse equivalence class of a box space, one can deduce the sequence of normal subgroups (from some index onwards).

In order to show this, we first prove that coarse fundamental groups of a Cayley graph of a quotient can detect the normal subgroup used to construct the quotient.

Given a presentation $G=\langle S| R \rangle$, and a normal subgroup $N\lhd G$, we denote by $|-|_{F_S}$ the length of relators in $R$ viewed as a subset of the free group $F_S$ on the set $S$ with its natural metric, and by $|-|_G$ the length of elements of $G$ in the Cayley graph $\Cay(G,S)$.

\begin{lemma}\label{detect}
Let $G$ be a finitely presented group with $G=\langle S| R \rangle$, let $k$ be equal to $\max\{|g|_{F_S}:g\in R\}$, let $N\lhd G$ with $2k< n = \inf\{|g|_G:g\in N\setminus\{e\}\}$, and let $r$ be a constant such that $2k\le 4r <n$.
Then $A_{1,r}(\Cay(G/N,\bar{S}),e)$ is isomorphic to $N$.
\end{lemma}

\begin{proof}
Define $\Phi\colon N\to A_{1,r}(\Cay(G/N),e)$ as follows: For $g\in N$ write $g$ as a word in the elements of $S$. This corresponds to a $1$-loop in $\Cay(G/N)$ based at $e$. We take $\Phi(g)$ equal to this loop (note that it is in particular an $r$-loop). 

To show that $\Phi$ is uniquely defined, suppose that $g\in N$ can be written in two ways as a word in the elements of $S$, say $v$ and $w$. Due to \cref{reduce} we can assume these words are reduced as elements of the free group $F_S$ on $S$. As $v$ and $w$ realize the same element in $G$ we can write $w^{-1}v = a$ where $a$ is in the normal subgroup of $F_S$ generated by the elements of $R$, so we can write $a = a_1a_2\ldots a_m$ with each $a_i$ the conjugate of an element in $R$.

As being \homotopic{r} is an equivalence relation, it suffices to show that the paths corresponding to the words $w$ and $whbh^{-1}$ are \homotopic{r} for every $b\in R$, every $h\in F_S$ and every word in $w$ in $F_S$ representing a loop in $G/N$.

Since $\ell(b)\le k \le 2r$, we have that $whbh^{-1}$ and $whh^{-1}$ correspond to \homotopic{r} loops due to \cref{jump}, and then by \cref{reduce} we have that these loops are \homotopic{r} to the one corresponding to $w$.

Thus $\Phi$ is well-defined. We now need to show that $\Phi$ is an isomorphism.

It is clearly a homomorphism, because of the correspondence between words in $F_S$ and paths in $\Cay(G/N)$.

To show that it is injective, it suffices to show that $\Phi(g)$ is not null-\homotopic{r} if $g\in N\setminus\{e\}$. Therefore we suppose there exists a element $g$ that does correspond to a null-\homotopic{r} loop. This means there exists a sequence $p_0$, $p_1$, \ldots, $p_n=\Phi(g)$, where $p_0$ is the trivial loop and $p_i$ is \close{r} to $p_{i+1}$ for every $i$.

Now as $2r< n$, there is a unique way of lifting an $r$-path in $\Cay(G/N)$ to an $r$-path in $\Cay(G)$, and
so we can take $\tilde{p}_i$ to be the lift of $p_i$ for every $i$. We know that $\tilde{p_0}(\ell(p_0))=e$ and $\tilde{p}_n(\ell(p_n))=g$. Now take $i$ to be the biggest value such that $\tilde{p}_i(\ell(p_i))=e$. Then $\tilde{p}_{i+1}(\ell(p_{i+1}))\neq e$ (since $\tilde{p}_n(\ell(p_n))=g$ and $g\neq e$).
However $\tilde{p}_{i+1}(\ell(p_{i+1}))\in N$, since $p_{i+1}$ is an $r$-loop in $\Cay(G/N)$, so $d(\tilde{p}_i(\ell(p_i)),\tilde{p}_{i+1}(\ell(p_{i+1})))\ge n> r$.

We know that $p_i$ is \close{r} to $p_{i+1}$ and there are two ways, (a) and (b), in which this can happen, according to the definition. If we were in the case of (a), then $p_i$ would be equal to $p_{i+1}$ on the interval where both of them are defined and constantly $e$ on the rest, and so then we would have that $e=\tilde{p}_i(\ell(p_i))=\tilde{p}_{i+1}(\ell(p_{i+1}))$. But we had assumed that $i$ is the biggest value for which $\tilde{p}_i(\ell(p_i))=e$, so the paths must be $r$-close as in (b), i.e. $p_i$ and $p_{i+1}$ are of the same length and $d(p_i(x),p_{i+1}(x))\le r$ for every $0\le x\le \ell(p_i)$.

Now $d(\tilde{p}_i(\ell(p_i)),\tilde{p}_{i+1}(\ell(p_i)))\ge n$ while $d(\tilde{p}_i(0),\tilde{p}_{i+1}(0))=0$.
Therefore there exists a $j$ such that $d(\tilde{p}_i(j),\tilde{p}_{i+1}(j))\le r$ and $d(\tilde{p}_i(j+1),\tilde{p}_{i+1}(j+1))> r$. However, since  $d(p_i(j+1),p_{i+1}(j+1))\le r$, there exists a $h\in N\setminus\{e\}$ such that $d(\tilde{p}_i(j+1),\tilde{p}_{i+1}(j+1)h)\le r$. 

We also note that $n \le d(\tilde{p}_{i+1}(j+1),\tilde{p}_{i+1}(j+1)h)$, since $d(\tilde{p}_{i+1}(j+1),\tilde{p}_{i+1}(j+1)h)= |h|_G$, and $h$ is a non-trivial element of $N$, so that $|h|_G>n$.

Now we can make the following computation:
\begin{eqnarray*}
n & \le & d(\tilde{p}_{i+1}(j+1),\tilde{p}_{i+1}(j+1)h)\\
& \le & d(\tilde{p}_{i+1}(j+1),\tilde{p}_{i+1}(j)) + d(\tilde{p}_{i+1}(j),\tilde{p}_i(j))\\
& & + d(\tilde{p}_i(j),\tilde{p}_i(j+1)) + d(\tilde{p}_i(j+1),\tilde{p}_{i+1}(j+1)h)\\
& \le & r + r + r + r\\
& \le & 4r < n.
\end{eqnarray*}
As this is impossible, we find that $\Phi(g)$ can not be null-\homotopic{r} and therefore $\Phi$ is injective.

To show that $\Phi$ is surjective, take an $r$-loop in $\Cay(G/N)$. Due to \cref{path} this loop is \homotopic{r} to a $1$-loop $p$. This loop corresponds to a word $w$ in $F(S)$ and this word gets mapped to an element $g$ in $G$ via the map $F_S\to G$. As $p$ is a loop, $g\in N$. Now by definition $\Phi(g) = p$.

So we can conclude that $\Phi$ is an isomorphism.
\end{proof}
We need the following lemma which tells us how $A_{1,r}$ behaves under quasi-isometries of the underlying graph.
Given two graphs $\mathcal{G}$ and $\mathcal{H}$, recall that $\phi: \mathcal{G} \rightarrow \mathcal{H}$ is a quasi-isometry with constant $C>0$ if for all $x,y \in \mathcal{G}$, 
$$d_{\mathcal{H}}(\phi(x),\phi(y)) \leq C d_{\mathcal{G}}(x,y) + C$$
and there exists a quasi-inverse map $\phi': \mathcal{H} \rightarrow \mathcal{G}$ which satisfies $d_{\mathcal{G}}(\phi'(x),\phi'(y)) \leq C d_{\mathcal{H}}(x,y) + C$ for all $x,y \in \mathcal{H}$,
$d_{\mathcal{G}}(x,\phi'(\phi(x)))\leq C$ for all $x\in \mathcal{G}$ and $d_{\mathcal{H}}(y,\phi(\phi'(y)))\leq C$ for all $y\in \mathcal{H}$.

\begin{lemma}\label{invar}
Let $C>0$ be a constant, let $\mathcal{G}\simeq_{\text{QI}}\mathcal{H}$ be two Cayley graphs that are quasi-isometric with constant $C$ and let $r\ge 2C$. Then there exists a homomorphism $ \Psi\colon A_{1,r}(\mathcal{G})\to A_{1,Cr+C} (\mathcal{H})$ that is surjective.
\end{lemma}
It will be very useful to assume that both the quasi-isometry and its quasi-inverse map the neutral element of one group to the neutral element of the other group. It is always possible to do this, as Cayley graphs have a natural isometric action of the group. The composition of the quasi-isometry and its quasi-inverse may then not be \close{C} to the identity map, but it will be \close{2C}.

\begin{proof}
Let $\phi\colon \mathcal{G}\to \mathcal{H}$ be a quasi-isometry with constant $C$. To construct the map $ \Psi\colon A_{1,r}(\mathcal{G})\to A_{1,Cr+C} (\mathcal{H})$, given a path $p$ in $\mathcal{G}$, we define $ \Psi(p)$ by $ \Psi(p)(i) = \phi(p(i))$ for every $0\le i\le \ell(p)$.
As we have assumed that $\phi(e)=e$, we have that $ \Psi(p)$ is a $(Cr+C)$-loop based at $e$.

In order to show that $ \Psi$ is well-defined, we also have to show that when $p$ and $q$ are $r$-loops in $\mathcal{G}$ that are \close{r}, then $ \Psi(p)$ and $ \Psi(q)$ are \close{(Cr+C)}. We can check this for the two cases, (a) and (b),  of being $r$-close.

In case (a), $p(i)=q(i)$ for every $0\le i\le \ell(p)$ and $q(i)=e$ for every $i\ge\ell(p)$, and so $ \Psi(p)(i)= \Psi(q)(i)$ for every $0\le i\le \ell(p)$ and $ \Psi(q)(i)=e$ for every $i\ge\ell(p)$.

In case (b), $\ell(p)=\ell(q)$ and $d(p(i),q(i))\le r$, and so  $d( \Psi(p)(i), \Psi(q)(i))=d(\phi(p(i)),\phi(q(i)))\le Cr+C$.
So $ \Psi(p)$ and $ \Psi(q)$ are \close{(Cr+C)}, whenever $p$ and $q$ are \close{r}.
Therefore $ \Psi$ is well-defined.

Now we will show that $ \Psi$ is surjective. To do so, take a $(Cr+C)$-loop $q$ in $\mathcal{H}$.
Due to \cref{path} $q$ is \homotopic{(Cr+C)} to a $1$-loop $\tilde{q}$.

As $\phi$ is a quasi-isometry, there exists a quasi-inverse $\phi'$. Now we can define an $r$-loop $p$ in $\mathcal{G}$ such that $p(i) = \phi'(\tilde{q}(i))$. This is an $r$-loop since $d(p(i),p(i+1)) = d(\phi'(\tilde{q}(i)),\phi'(\tilde{q}(i+1))) \le 2C\le r$ (as $\tilde{q}$ is a 1-loop and so $d(\tilde{q}(i),\tilde{q}(i+1))\leq 1$).

It suffices to show that $ \Psi(p)$ is \homotopic{(Cr+C)} to $\tilde{q}$. By definition we have that $ \Psi(p)(i) = \phi(p(i)) =  \phi(\phi'(\tilde{q}(i)))$ for every $0\le i\le\ell(p)$, so $d( \Psi(p)(i),\tilde{q}(i))\le 2C\le Cr+C$.
So $ \Psi(p)$ is \homotopic{(Cr+C)} to $q$. Therefore $ \Psi$ is surjective.
\end{proof}

Now we are ready to prove the main result, that coarsely equivalent box spaces must be constructed using essentially the same sequence of normal subgroups.

\begin{theorem}[\cref{main} in Introduction]
Let $G$ be a finitely presented group and $H$ a finitely generated group, with respective filtrations $N_i$ and $M_i$ such that $\Box_{(N_i)}G\simeq_{\text{CE}} \Box_{(M_i)}H$. Then there exists an almost permutation with bounded displacement $f$ of $\N$ such that $N_i$ is isomorphic to $M_{f(i)}$ for every $i$ in the domain of $f$.
\end{theorem}

\begin{proof}
We first note that $H$ is also finitely presented: the box spaces being coarsely equivalent implies by the main theorem of \cite{AA} that the groups must be quasi-isometric, and finite presentability is invariant under quasi-isometry.

As $\Box_{(N_i)}G\simeq_{\text{CE}} \Box_{(M_i)}H$, we know due to Lemma 1 of \cite{AA} that there exists an almost permutation of $\N$ with bounded displacement and a constant $C$ such that $G/N_i$ is quasi-isometric, with constant $C$, to $H/M_{f(i)}$ for every $i$ in the domain of $f$.

Now take finite presentations of $G=\langle S,R\rangle$ and $H=\langle S',R'\rangle$, where $S$ and $S'$ are the generating sets used for the construction of the box spaces. Also take $k=\max\left(\{|g|: g\in R\cup R'\}\cup \{2C\}\right)$ and take $I$ such that $\inf\{|g|: g\in N_i\setminus\{e\}\}>Ck+C$ and $\inf\{|g|:g\in M_{f(i)}\setminus\{e\}\}>Ck+C$ for every $i\ge I$. Such an $I$ exists, because $N_i$ and $M_i$ are filtrations and $f$ is an almost permutation of $\N$ with bounded displacement.

By combining \cref{detect} and \cref{invar} we find for any $i\ge I$ that $$N_i\cong A_{1,k} \left(\Cay(G/N_i)\right) \twoheadrightarrow A_{1,Ck+C} \left(\Cay(H/M_{f(i)})\right) \cong M_{f(i)}.$$
Similarly we find for any $i\ge I$ that $$M_{f(i)}\cong A_{1,k} \left(\Cay(H/M_{f(i)})\right) \twoheadrightarrow A_{1,Ck+C} \left(\Cay(G/N_i)\right) \cong N_i.$$
Combining these maps provides a surjective homomorphism $N_i\twoheadrightarrow N_i$. Now $N_i\lhd G$ is residually finite, therefore it is Hopfian. This means that this surjective homomorphism is also injective. Therefore $N_i\cong M_{f(i)}$, when $i\ge I$. Now we can remove all $i< I$ from the domain of $f$. Then $f$ is still an almost permutation of $\N$ with bounded displacement and $N_i\cong M_{f(i)}$ for every $i$ in the domain of $f$. 
\end{proof}

\section{Applications}
In this section we look at some applications of \cref{main}.
\subsection{Bounded covers}
A first application of \cref{main}, is that there exist two box spaces $\Box_{(N_i)}G\not\simeq_{\text{CE}}\Box_{(M_i)}G$ of the same group $G$ such that $G/N_i\twoheadrightarrow G/M_i$ with $[M_i:N_i]$ bounded. This is surprising, because for two groups $G\twoheadrightarrow H$ with finite kernel, we have that $G$ and $H$ are quasi-isometric.

An example of such box spaces can be found for the free group. Consider $F_2$ as an index 12 subgroup of $\SL{2}{\Z}$, where the generating set is taken to be $$\left\{\begin{bmatrix}1&2\\0&1\end{bmatrix}, \begin{bmatrix}1&0\\2&1\end{bmatrix} \right\}.$$ 
Then take $N_i$ to be the kernel of $F_2\to \SL{2}{\Z/4^i\Z}$ and take $M_i$ to be the kernel of  $F_2\to \SL{2}{\Z/2^{2i+1}\Z}$. 
For $i\geq 2$, we have that the kernel of $\SL{2}{\Z} \to \SL{2}{\Z/2^i\Z}$ is contained in $F_2$.
Thus for $i\ge 1$ we have $[F_2:N_i]=2^{6i-4}$, while $[F_2:M_i] = 2^{6i-1}$. As they are subgroups of $F_2$, $N_i$ and $M_i$ are free groups for every $i$, and we can calculate the rank for any $N\lhd F_2$ using the Nielsen-Schreier rank formula, $\rk{N} = [F_2:N] + 1$.

If we assume that $\Box_{(N_i)}F_2\simeq_{\text{CE}}\Box_{(M_i)}F_2$, then due to \cref{main} there exist some $i$ and $j$ in $\N$ such that $N_i\cong M_j$. So $2^{6i-4}+1 = \rk{N_i} =\rk{M_j} =2^{6j-1} + 1$. So $6i-4=6j-1$ which is impossible if $i$ and $j$ are both in $\N$.

We have thus proved the following.
\begin{theorem}[\cref{bounded} in Introduction]
There exist two box spaces $\Box_{(N_i)}G\not\simeq_{\text{CE}}\Box_{(M_i)}G$ of the same group $G$ such that $G/N_i\twoheadrightarrow G/M_i$ with $[M_i:N_i]$ bounded.
\end{theorem}

\subsection{Rigidity of box spaces}
Another consequence of \cref{main} is that for finitely presented groups we can strengthen Theorem 7 of \cite{AA}, which states that two box spaces that are coarsely equivalent stem from quasi-isometric groups.
\begin{theorem}[\cref{ComInt} in Introduction]\label{Com}
Let $G$ and $H$ be finitely presented groups, then there exist two filtrations $N_i$ and $M_i$ of $G$ and $H$ respectively such that $\Box_{(N_i)} G$ and $\Box_{(M_i)} H$ are coarsely equivalent if and only if $G$ is commensurable to $H$ via a normal subgroup.
\end{theorem}
Note that two groups are commensurable via a normal subgroup if there exist finite index normal subgroups of each of the groups such that these subgroups are isomorphic.

One of the implications of this theorem follows immediately from \cref{main}. For the other direction, we first recall the following well-known result.
\begin{proposition}\label{ComNorm}
If $G$ is residually finite, then it has a filtration consisting of characteristic subgroups.
\end{proposition}
\begin{proof}
As $G$ is residually finite, it has a filtration $N_i$. Now take $M_i = \cap_{\varphi}\varphi(N_i)$, where the intersection is taken over all automorphisms $\varphi$ in $\operatorname{Aut}(G)$.
The subgroups $M_i$ are characteristic, since for any $\psi\in \operatorname{Aut}(G)$, $\psi(M_i) = \psi\left(\cap_{\varphi}\varphi(N_i)\right) = \cap_{\varphi}(\psi\varphi)(N_i)=M_i$.

Now $M_i$ is a filtration of $G$: $M_i$ is of finite index in $G$ for every $i$ since there are finitely many possibilities for $\varphi(N_i)$ as $G$ is finitely generated, and thus $M_i$ is the intersection of finitely many finite index subgroups.
\end{proof}

We can use this proposition to prove \cref{Com}.

\begin{proof}[Proof of \cref{Com}]
If $\Box_{(N_i)} G$ is coarsely equivalent to $\Box_{(M_i)} H$, then due to \cref{main} we have that for a large enough $i$ there exists a $j$ such that $N_i$ is isomorphic to $M_j$. As both $N_i$ and $M_j$ are normal subgroups, this proves one direction of the theorem.

Conversely, if $G$ and $H$ are commensurable via normal subgroups, they both have finite index normal subgroups that are isomorphic to one another. In fact we may say that they have a common finite index normal subgroup $K$.
As $G$ is residually finite, we have that $K$ is also residually finite. Now due to \cref{ComNorm} we can take a filtration $N_i$ of $K$ consisting of characteristic subgroups. Now as $K$ is a finite index normal subgroup of $G$ and $H$, we have that $N_i$ is a filtration of $G$ and of $H$, and the corresponding box spaces are coarsely equivalent.
\end{proof}

We remark that this result is in some sense optimal, since by Proposition 10 of \cite{AA}, there exist non-commensurable groups which admit isometric box spaces. The groups used are the wreath products $\mathbb{Z}/4 \mathbb{Z} \wr \mathbb{Z}$ and $(\mathbb{Z}/2 \mathbb{Z} \times \mathbb{Z}/2 \mathbb{Z}) \wr \mathbb{Z}$, which are ``easy'' examples of groups which are not fnitely presented. 

This result also allows us to answer the question posed at the end of \cite{Das}, which asks whether residually finite groups that are uniformly measure equivalent necessarily admit coarsely equivalent box spaces. The answer can be seen to be negative, by taking two residually finite, finitely presented, uniformly measure equivalent groups which are not commensurable (for example, cocompact lattices in $SL(2,\mathbb{C})$ are a source of such examples). 

\subsection{Rank gradient and first $\ell^2$ Betti number}
From \cref{main} we can deduce that non-zero rank gradient of the group is a coarse invariant for box spaces of finitely presented groups. 

For a finitely generated, residually finite group $G$, the rank gradient $\RG{G}{(N_i)}$ of $G$ with respect to a filtration $(N_i)$ is defined by
$$\RG{G}{(N_i)}:= \lim_{i\rightarrow \infty} \frac{\rk{N_i}-1}{[G:N_i]},$$
where $\rk{N}$ denotes the rank of $N$, i.e. the minimal cardinality of a generating set. This notion was first introduced by Lackenby in \cite{Lac}, and is connected in interesting ways to analytic properties of $G$, see for example \cite{AJN}.

\begin{theorem}[\cref{RGInt} in Introduction]
Given two finitely presented, residually finite groups $G$ and $H$ and respective filtrations $(N_i)$ and $(M_i)$, if $\Box_{(N_i)}G$ is coarsely equivalent to $\Box_{(M_i)}H$ then $\RG{G}{(N_i)}>0$ if and only if $\RG{H}{(M_i)}>0$. 
\end{theorem}

\begin{proof}
We begin with the remark that `` $\operatorname{rk}-1$ '' is submultiplicative, that is, if $\Lambda<\Gamma$, then
$$\rk{\Lambda}-1\leq (\rk{\Gamma}-1)[\Gamma:\Lambda].$$
This inequality can be deduced from the Nielsen-Schreier formula applied to the free group $F_{\rk{G}}$ on $\rk{G}$ letters, and the subgroup $K< F_{\rk{G}}$ such that $H$ is the image of $K$ under the surjective homomorphism $F_{\rk{G}}\rightarrow G$. Note that $H$ may be generated by fewer elements than $\rk{K}$. 
Thus the limit in the definition is an infimum, and for all $i$,
$$\frac{\rk{N_i}-1}{[G:N_i]}\leq \frac{\rk{N_{i-1}}-1}{[G:N_{i-1}]} \text{ and } \frac{\rk{M_i}-1}{[H:M_i]}\leq \frac{\rk{M_{i-1}}-1}{[H:M_{i-1}]}.$$
By \cref{main} and Lemma 1 of \cite{AA}, there is an almost permutation $f: \N \rightarrow \N$ with bounded displacement such that $N_i\cong M_{f(i)}$ for every $i$ in the domain of $f$, and such that there is a constant $C>0$ with $G/N_i$ and $H/M_{f(i)}$ being $C$-quasi-isometric for all $i$ in the domain of $f$. 
Thus we have
$$\inf_{i} \frac{\rk{N_i}-1}{[G:N_i]} = \inf_i \frac{\rk{M_{f(i)}}-1}{[G:N_i]}\geq \inf_i \frac{\rk{M_i}-1}{|B_G(C^2)|\cdot[H:M_i]},$$
where $|B_G(C^2)|$ denotes the cardinality of a ball of radius $C^2$ in $G$. So we have that $\RG{H}{(M_i)}>0$ implies that $\RG{G}{(N_i)}>0$, and the theorem follows.
\end{proof}

As another corollary of \cref{main}, we have that having the first $\ell^2$ Betti number equal to zero is a coarse invariant of box spaces of finitely presented groups. For the definition and properties of the first $\ell^2$ Betti number, we refer the reader to \cite{Lu02}. By the L\"{u}ck Apprroximation Theorem (\cite{Lu94}), we have that for a finitely presented group $G$, the first $\ell^2$ Betti number $\beta_1^{(2)}(G)$ can be approximated using a filtration $(N_i)$ of $G$ as follows:
$$\beta_1^{(2)}(G) = \lim_{i\rightarrow \infty} \frac{b_1(N_i)}{[G:N_i]},$$
where $b_1(N)$ denotes the classical first Betti number of $N$. 

Thus, by the same argument as above, we have 
$$\lim_{i\rightarrow \infty} \frac{b_1(N_i)}{[G:N_i]} = \lim_{i\rightarrow \infty} \frac{b_1(M_{f(i)})}{[G:N_i]} \geq \lim_{i\rightarrow \infty} \frac{b_1(M_{i})}{|B_G(C^2)|\cdot[H:M_i]},$$
under the same hypotheses as above, and so we recover the following theorem.

\begin{theorem}[\cref{l2Int} in Introduction]
Given two finitely presented, residually finite groups $G$ and $H$ and respective filtrations $(N_i)$ and $(M_i)$, if $\Box_{(N_i)}G$ is coarsely equivalent to $\Box_{(M_i)}H$ then $\beta_1^{(2)}(G)>0$ if and only if $\beta_1^{(2)}(H)>0$.
\end{theorem}

Note that a more general result is Corollary 1.3 of \cite{Das}: if $\Box_{(N_i)}G$ is coarsely equivalent to $\Box_{(M_i)}H$, even without the assumption of finite presentation, then $G$ and $H$ are uniformly measure equivalent and thus have proportional $\ell^2$ Betti numbers. 

\subsection{Box spaces of free groups}

The main theorem also allows us to easily distinguish box spaces of the free group.

\begin{theorem}[\cref{HilbInt} in Introduction]
For each $n\geq 2$, there exist infinitely many coarse equivalence classes of box spaces of the free group $F_n$ that coarsely embed into a Hilbert space.
\end{theorem}

\begin{proof}
Given $m\geq 2$, consider the $m$-homology filtration $(N_i)$, defined inductively by $N_1:= F_n$, $N_{i+1}= N_i^m [N_i,N_i]$ (that is, $N_{i+1}$ is generated by commutators and $m$th powers of elements of $N_i$). 
Since for each $j$, $N_j/N_{j+1}\cong \mathbb{Z}_m^{\rk{N_j}}$, using the Nielsen-Schreier rank formula we can deduce that
$$\rk{N_i}= m^{\sum{\rk{N_j}}}(n-1)+1,$$
where the sum in the exponent runs from $j=1$ to $j=i-1$. 
By consideration of these ranks for coprime $m$ and $k$, we see that the corresponding $m$-homology and $k$-homology filtrations $(N_i)$ and $(M_i)$ will satisfy $N_i \not\cong M_j$ for all $i,j$ sufficiently large.
Thus, considering $m$-homology filtrations for various pairwise coprime $m$ gives rise to box spaces of $F_n$ which are not coarsely equivalent by \cref{main}.
Such box spaces are coarsely embeddable into a Hilbert space by the main result of \cite{Khu}.
\end{proof}

We now prove \cref{InfRam} from the introduction. Note that since being a Ramanujan expander is not preserved by coarse equivalences, the relevant question becomes how many coarse equivalence classes of box spaces there are such that each equivalence class contains at least one box space which is a Ramanujan expander.

\begin{theorem}[\cref{InfRam} in Introduction]\label{RamThm}
There exist infinitely many coarse equivalence classes of box spaces of the free group $F_3$ that contain Ramanujan expanders.
\end{theorem}

\begin{proof}
We begin by fixing an odd prime $q$ such that $-1$ is a quadratic residue modulo $q$, and $5$ is a quadratic residue modulo $2q$. 
By Theorem 7.4.3 of \cite{Lub}, there exists a filtration $(N_i)$ of $F_3$ with the property that the quotients $F_3/Ni \cong \operatorname{PSL}_2(q^i)$ form a Ramanujan expander sequence, and such that $N_i/N_{i+1} \cong \mathbb{Z}_q^3$ (see \cite{DK}). 
Note that $[F_3:N_1]=|\operatorname{PSL}_2(q)|=q(q^2-1)/2$.
Using this, and the Nielsen-Schreier rank formula, we obtain
$$\rk{N_i}= 2q^{3(i-1)}\cdot \frac{q(q^2-1)}{2}+1= (q^2-1)q^{3i-2}+1.$$
By taking such expanders corresponding to different primes $q$ satisfying the above conditions, consideration of the ranks mean that we obtain box spaces of $F_3$ which are not coarsely equivalent by \cref{main}.

To see that infinitely many such $q$ exist, we proceed as follows.
It is well-known that $-1$ is a square modulo $q$ if and only if $q$ is congruent to $1$ modulo $4$.
Now $5$ is a square modulo $2q$ if and only if $5$ is a square modulo $q$ by the Chinese remainder theorem, and this happens if and only if $q$ is a square modulo $5$ by quadratic reciprocity, which in turn happens if and only if $q$ is congruent to $1$ or $-1$ modulo $5$. 
If $q$ is congruent to $1$ modulo $20$ then $q$ is both congruent to $1$ modulo $4$ and congruent to $1$ modulo $5$, and by Dirichlet's Theorem, there are infinitely many such primes $q$.
\end{proof}
See also \cite{Hum}, where a continuum of regular equivalence classes of expanders with large girth (i.e. the length of the smallest loop tending to infinity) is constructed.

\begin{theorem}[\cref{coprimeInt} in Introduction]
Given $n \geq 3$, there exists a box space of the free group $F_n$ such that no box space of $F_m$ with $m-1$ coprime to $n-1$ is coarsely equivalent to it.
\end{theorem}

\begin{proof}
Consider the $(n-1)$-homology filtration $(N_i)$ of $F_n$, recalling that this filtration is defined inductively by $N_1:=F_n$, $N_{i+1}:=N_i^{n-1}[N_i,N_i]$.
If $(M_i)$ is a filtration of $F_m$, then by \cref{main}, if $\Box_{(N_i)}F_n \simeq_{CE} \Box_{(M_i)}F_m$ then for some $i, j$, we must have $$(n-1)^a +1 = [F_m:M_j](m-1) +1$$ with $a\in \mathbb{N}$ (by rank considerations). But this is impossible due to the assumptions on $m$. 
\end{proof}

\subsection{Induced box spaces of subgroups}
Given a residually finite, finitely generated group $G$ with filtration $(M_i)$, and a subgroup $H<G$, one can construct a filtration of $H$ by taking the intersections $(M_i \cap H)$. A natural question is how the box spaces $\Box_{(M_i)}G$ and $\Box_{(M_i\cap H)}H$ are related. 
Note that the quotient $H/(M_i\cap H)$ in $\Box_{(M_i\cap H)}H$ is isomorphic to $HM_i/M_i$, which can be viewed as a subgroup of $G/M_i$. 

A situation of particular interest is when $H$ is a finite index subgroup of $G$, as this implies that $G$ and $H$ are quasi-isometric. We now show that this surprisingly does not necessarily mean that the box spaces $\Box_{(M_i)}G$ and $\Box_{(M_i\cap H)}H$ are coarsely equivalent, correcting Proposition 4 of \cite{Kh12}\footnote{The error in Proposition 4 of \cite{Kh12} arises from assuming that the metric induced on the $H/(M_i\cap H)\cong HM_i/M_i$ by considering them as subgroups of $G/M_i$ is coarsely equivalent to a metric coming from a fixed generating set of $H$, which is not the case. The result is true if this other induced metric is considered. The main results of the paper, Theorem 10 and Theorem 13, are unaffected.}.

\begin{theorem}[\cref{CorInt} in Introduction]
There exists a finitely generated, residually finite group $G$ with filtration $(M_i)$ and a normal finite index subgroup $H\lhd G$ such that $\Box_{(M_i)}G$ and $\Box_{(M_i\cap H)}H$ are not coarsely equivalent.
\end{theorem}

\begin{proof}
We will take $G$ to be the free group $F_3$. For convenience, we will use a lattice of normal subgroups of $F_3$ constructed in \cite{DK}; we briefly recall the relevant details and refer the reader to \cite{DK} for the construction.

We first fix the odd prime $q=29$, noting that $-1$ is a quadratic residue modulo $q$, and $5$ is a quadratic residue modulo $2q$. 
There exists a filtration $(N_i)$ of $F_3$ such that $[F_3:N_i]= q^{3i-2}(q^2-1)/2$ (see \cite{DK}) and so by the Nielsen-Schreier formula, $\rk{N_i}= (q^2-1)q^{3i-2}+1$. 

For the group $H$, we will take the subgroup $\Gamma(N_3)$ of $F_3$. This is the subgroup corresponding to the $q$-homology cover of the finite quotient $F_3/N_3$, defined by $\Gamma(N_3):= N_3^q [N_3,N_3]$ (i.e. the subgroup generated by all commutators and all $q$th powers of elements of $N_3$). We thus have that $[N_3:\Gamma(N_3)]=q^{\rk{N_3}}=q^{(q^2-1)q^7+1}$. Note that $\Gamma(N_3)$ is of finite index in $F_3$.

The filtration of $G=F_3$ we will consider is the filtration $(N_i)$, starting at $i=4$. So, the box spaces in question are $\Box_{(N_i)}F_3$ and $\Box_{(N_i\cap \Gamma(N_3))} \Gamma(N_3)$. We have that $\Gamma(N_3)/(N_i\cap\Gamma(N_3))$ is isomorphic to $N_i \Gamma(N_3) / N_i$, which in turn is isomorphic to $N_4/N_i$ by Proposition 3.9 of \cite{DK}.

We can now compute the rank of the subgroups $N_i\cap \Gamma(N_3)$ using the Nielsen-Schreier formula:
\begin{align*}
\rk{N_i\cap \Gamma(N_3)}&= 2[F_3:N_i\cap \Gamma(N_3)]+1\\
&= 2[F_3:N_3]\cdot [N_3:\Gamma(N_3)] \cdot [\Gamma(N_3):N_i\cap \Gamma(N_3)]+1\\
&= q^{7}(q^2-1) \cdot q^{(q^2-1)q^7+1} \cdot [N_4:N_i]+1\\
&= (q^2-1)q^{(q^2-1)q^7+8} \cdot q^{3(i-4)} +1\\
&= (q^2-1)q^{(q^2-1)q^7-4+3i}+1.
\end{align*}
Now if $\Box_{(N_i)}F_3$ and $\Box_{(N_i\cap \Gamma(N_3))} \Gamma(N_3)$ were coarsely equivalent, then by \cref{main} for some $i,j\in \mathbb{N}$, comparing ranks of $N_i\cap \Gamma(N_3)$ and $N_j$, we must have that $(q^2-1)q^{(q^2-1)q^7-4+3i}+1= (q^2-1)q^{3j-2}+1$. 
This implies that $(q^2-1)q^7-4+3i=3j-2$, which means that $(q^2-1)q^7-2=(29^2-1)\cdot 29^7-2$ must be divisible by $3$. Since this is not the case, we conclude that $\Box_{(N_i)}F_3$ and $\Box_{(N_i\cap \Gamma(N_3))} \Gamma(N_3)$ are not coarsely equivalent.
\end{proof}


\begin{thebibliography}{[WWWW]}
\bibitem[AJN]{AJN}
M. Ab\'{e}rt, A. Jaikin-Zapirain and N. Nikolov, \emph{The rank gradient from a combinatorial viewpoint}, Groups Geom. Dyn. 5 (2011)
\bibitem[AFS]{AFS}
V. Alekseev and M. Finn-Sell, \emph{Sofic boundaries of groups and coarse geometry of sofic approximations}, preprint, arxiv:1608.02242 (2016)
\bibitem[BBLL]{BBLL}
E. Babson, H. Barcelo, M. de Longueville and R. Laubenbacher, \emph{Homotopy theory of graphs}, J. Alg. Comb. 24 (2006)
\bibitem[BCW]{BCW}
H. Barcelo, V. Capraro and  J. A. White, \emph{Discrete homology theory for metric spaces}, Bull. Lond. Math. Soc. 46 (2014)
\bibitem[BKLW]{BKLW}
H. Barcelo, X. Kramer, R. Laubenbacher and C. Weaver, \emph{Foundations of a connectivity theory for simplicial complexes}, Advances in Applied Math. 26 (2001)
\bibitem[BL]{BL}
H. Barcelo, R. Laubenbacher, \emph{Perspectives in A-homotopy theory and its applications}, Discrete Mathematics 298 (2005)
\bibitem[CWW]{CWW}
X. Chen, Q. Wang and X. Wang, \emph{Characterization of the Haagerup property by fibred coarse embedding into Hilbert space}, Bull. Lond. Math. Soc. 45 (2013)
\bibitem[Das]{Das}
K. Das, \emph{From the geometry of box spaces to the geometry and measured couplings of groups}, arXiv:1512.08828 (2015)
\bibitem[DK]{DK}
T. Delabie and A. Khukhro, \emph{Box spaces of the free group that neither contain expanders nor embed into a Hilbert space}, preprint, arXiv:1611.08451 (2016)
\bibitem[Hat]{Hat}
A. Hatcher, \emph{Algebraic Topology}, Cambridge University Press (2002)
\bibitem[Hum]{Hum}
D. Hume, \emph{A continuum of expanders}, Fundamenta Mathematicae, to appear
\bibitem[Kh12]{Kh12}
A. Khukhro, \emph{Box spaces, group extensions and coarse embeddings into Hilbert space}, J. Funct. Anal. 263 (2012)
\bibitem[Khu]{Khu}
A. Khukhro, \emph{Embeddable box spaces of free groups}, Math. Ann. 360 (2014)
\bibitem[KV]{AA}
A. Khukhro and A. Valette, \emph{Expanders and box spaces}, preprint, arXiv:1509.01394 (2015)
\bibitem[Lac]{Lac}
M. Lackenby, \emph{Expanders, rank and graphs of groups}, Israel J. Math. 146 (2005)
\bibitem[Lub]{Lub}
A. Lubotzky, \emph{Discrete groups, expanding graphs and invariant measures}, Birkh\"{a}user Progress in Mathematics Vol. 125 (1994)
\bibitem[Lu94]{Lu94}
W. L\"{u}ck, \emph{Approximating L2-invariants by their finite-dimensional analogues}, Geom. Funct. Anal. 4 (1994)
\bibitem[Lu02]{Lu02}
W. L\"{u}ck, \emph{L2-invariants: theory and applications to geometry and K-theory}, Ergebnisse der Mathematik und ihrer Grenzgebiete (3 Series) 44, Springer-Verlag (2002)
\bibitem[Mar]{Mar}
G. Margulis, \emph{Explicit constructions of expanders}, Prob. Pered. Inform. 9 (1973)
\bibitem[NY]{NY}
P. Nowak and G. Yu, \emph{Large Scale Geometry}, EMS Textbooks in Mathematics, EMS (2012)
\bibitem[Roe]{Roe}
J. Roe, \emph{Lectures on Coarse Geometry}, AMS University Lecture Series 31, AMS (2003)
\bibitem[WY]{WY}
R. Willett and G. Yu, \emph{Higher index theory for certain expanders and Gromov monster groups II} Adv. Math. 229 (2012)


\end{thebibliography}

\end{document}